\documentclass[12pt,reqno]{amsart}

\usepackage[latin1]{inputenc}
\usepackage{amsmath}
\usepackage{amsfonts}
\usepackage{amssymb}
\usepackage{graphics}
\usepackage{enumerate}
\usepackage{amssymb,amsmath,amsthm,amscd,latexsym,verbatim,graphicx,amsfonts}

\usepackage{amscd}
\usepackage{amsmath}
\usepackage{amssymb}
\usepackage{amsthm}
\usepackage{latexsym}
\usepackage{verbatim}

\theoremstyle{plain}
\newtheorem{theorem}{Theorem}[section]

\newtheorem{corollary}[theorem]{Corollary}
\newtheorem{lemma}[theorem]{Lemma}
\newtheorem{prop}[theorem]{Proposition}
\theoremstyle{definition}

\theoremstyle{remark}

\newcommand{\bbC}{\mathbb{C}}

\newcommand{\bbD}{\mathbb{D}}

\newcommand{\bbN}{\mathbb{N}}

\newcommand{\mco}{\mathcal{O}}
\newcommand{\mcp}{\mathcal{P}}

\DeclareMathOperator*{\Real}{Re}

\title[]{Hyponormal Toeplitz Operators with Non-harmonic Algebraic Symbol}
\author[]{Brian Simanek}
\date{}

\begin{document}
\maketitle

\begin{abstract}
Given a bounded function $\varphi$ on the unit disk in the complex plane, we consider the operator $T_{\varphi}$, defined on the Bergman space of the disk and given by $T_{\varphi}(f)=P(\varphi f)$, where $P$ denotes the projection to the Bergman space in $L^2(\bbD,dA)$.  We provide new necessary conditions on $\varphi$ for $T_{\varphi}$ to be hyponormal, extending recent results of Fleeman and Liaw.  One of our main results provides a necessary condition on the complex constant $C$ for the operator $T_{z^n+C|z|^s}$ to be hyponormal.  This condition is also sufficient if $s\geq2n$.
\end{abstract}

\vspace{4mm}


\vspace{2mm}


\vspace{2mm}

\normalsize

\section{Introduction}\label{intro}

Let $dA$ denote normalized area measure on the unit disk and let $A^2(\bbD)$ denote the Bergman space of the unit disk, that is
\[
A^2(\bbD)=\left\{f:\int_{\bbD}|f|^2dA<\infty,\, f\mathrm{\, is\, analytic\, in\, }\bbD\right\}
\]
An alternative description of $A^2(\bbD)$ is
\[
A^2(\bbD)=\left\{f(z)=\sum_{n=0}^{\infty}a_nz^n:\sum_{n=0}^{\infty}\frac{|a_n|^2}{n+1}<\infty\right\}
\]

A bounded operator $T$ is said to be \textit{hyponormal} if $[T^*,T]\geq0$, where $T^*$ denotes the adjoint of $T$.  An equivalent definition of hyponormality is $\|Tu\|\geq\|T^*u\|$ for all vectors $u$.  Such operators are of interest because of Putnam's inequality (see \cite[Theorem 1]{Putnam}), which says that hyponormal operators satisfy
\[
\|[T^*,T]\|\leq\frac{|\sigma(T)|}{\pi}
\]
where $\sigma(T)$ is the spectrum of $T$ and $|\cdot|$ dentoes the two-dimensional area.

The operators we are interested in are the Toeplitz operators with symbol $\varphi\in L^{\infty}(\bbD)$.  More precisely, if $\varphi\in L^{\infty}(\bbD)$, then we define the operator $T_{\varphi}:A^2(\bbD)\rightarrow A^2(\bbD)$ by
\[
T_{\varphi}(f)=P(\varphi f),
\]
where $P$ denotes the projection to the Bergman space in $L^2(\bbD,dA)$.  We are interested in understanding what symbols $\varphi$ yield Toeplitz operators $T_{\varphi}$ that are hyponormal.  An analogous question can be asked in the setting of the Hardy space of the unit disk and it was answered by Cowen in \cite{Cowen}.

There are several obvious examples of hyponormal Toeplitz operators acting on the Bergman space.  For instance, $T_{|z|^2}$ is hyponormal because (recalling the fact that $T_{\varphi}^*=T_{\overline{\varphi}}$) it is self-adjoint.  The operator $T_{z}$ is also hyponormal because if $f\in A^2(\bbD)$, then
\[
\|T_zf\|^2=\int_{\bbD}|zf|^2dA=\int_{\bbD}|\bar{z}f|^2dA\geq\int_{\bbD}|P(\bar{z}f)|^2dA=\|T^*_{z}f\|^2,
\]
The same reasoning shows that $T_g$ is hyponormal for any $g\in H^{\infty}(\bbD)$.

While a complete characterization of hyponormal Toeplitz operators acting on the Bergman space has remained elusive, there has been a substantial amount of work on understanding the case when $\varphi$ is a polynomial in $z$ and $\bar{z}$ (see \cite{CC,FL,Hwang,Hwang2,LuShi}).  The main focus of this work will be to understand how one can perturb a hyponormal operator in a way that preserves hyponormality.  For example, one could ask the following question:
\begin{itemize}
\item[(Q-I)]  If $m,n\in\bbN$, for what values of $a\in\bbC$ is $T_{z^m+a\bar{z}^n}$ hyponormal?
\end{itemize}

This question was answered completely by Sadraoui in \cite[Proposition 1.4.4]{Sadraoui}.  Since $T_g$ is hyponormal whenever $g\in H^{\infty}(\bbD)$, it is no surprise that Sadraoui's result tells us that $T_{z^m+a\bar{z}^n}$ is hyponormal if and only if $|a|$ is sufficiently small (where ``sufficiently small" depends on $n$ and $m$).  What is perhaps more surprising is the answer to the following question:

\begin{itemize}
\item[(Q-II)]  For what values of $a\in\bbC$ is $T_{(z-1)^2+a\bar{z}}$ hyponormal?
\end{itemize}

The surprising answer is that $T_{(z-1)^2+a\bar{z}}$ is hyponormal if and only if $a=0$.  This follows from the following result, which is \cite[Theorem 1.4.3]{Sadraoui}.

\begin{theorem}\label{sad143}(\cite[Theorem 1.4.3]{Sadraoui})
Suppose $f,g\in H^{\infty}(\bbD)$ and $f'\in H^2(\bbD)$.  If $T_{f+\bar{g}}$ is hyponormal, then $g'\in H^2(\bbD)$ and $|g'|\leq|f'|$ almost everywhere on $\partial\bbD$.
\end{theorem}

Theorem \ref{sad143} tells us that when trying to understand the effect of a perturbation of the symbol $f$, one must measure the perturbation not just of $f$, but of $f'$ as well.  This is a very important insight and this theorem has many powerful consequences.  For instance, if one applies this result with $f(z)=\alpha z^m+\beta z^n$ and $g(z)=\gamma z^p+\delta z^q$, then one obtains a new short proof of \cite[Theorem 3.10]{CC}.  If one applies this result with $f(z)=\gamma z+7z^2+2z^3$ and $g(z)=8z^3+z^2+\beta z$ with $|\gamma|=|\beta|$, then one obtains a new proof of \cite[Example 2.9]{Hwang2}.  While Sadraoui's result (Theorem \ref{sad143}) is very powerful, a more general result was proven in \cite{AC}.

Here are some additional questions that we will address in the sections that follow:

\begin{itemize}
\item[(Q-III)]  If $n\in\bbN$ and $s\in(0,\infty)$, for what values of $a\in\bbC$ is $T_{z^n+a|z|^s}$ hyponormal?
\item[(Q-IV)]  If $m<n\in\bbN$ and $s,t\in[0,\infty)$, for what values of $a\in\bbC$ is $T_{z^n|z|^s+az^m|z|^t}$ hyponormal?
\item[(Q-V)]  If $n\in\bbN$ and $s_0,s_1\in(0,\infty)$, for what values of $a\in\bbC$ is $T_{z^n\left(|z|^{s_0}+a_1|z|^{s_1}\right)}$ hyponormal?
\end{itemize}

The remainder of the paper is devoted to proving results that will help us answer these three questions.  We will provide a partial answer to question (Q-III) in Section \ref{Ex1}, which is a complete answer in the case $s\geq2n$.  In Section \ref{additive} we will present results related to question (Q-IV) and in Section \ref{multiplicative} we will present results related to question (Q-V).

A helpful formula that we will repeatedly use throughout this paper comes from \cite{FL} and is given by
\begin{align}\label{expand}
 &\left\langle \left[(T+S)^{*},T+S\right]u,u\right\rangle \nonumber \\
 &\quad =\left\langle Tu,Tu\right\rangle -\left\langle T^{*}u,T^{*}u\right\rangle +2\mathrm{Re}\left[\left\langle Tu,Su\right\rangle -\left\langle T^{*}u,S^{*}u\right\rangle \right]+\left\langle Su,Su\right\rangle -\left\langle S^{*}u,S^{*}u\right\rangle.
\end{align}
This formula will enable us to isolate the perturbations and understand their effect on hyponormality.

\section{The Operator $T_{z^n+C|z|^{s}}$}\label{Ex1}

In this section, we will answer the question (Q-III).  In \cite{FL} Fleeman and Liaw consider non-harmonic polynomials and present the somewhat surprising example that $T_{z+C|z|^2}$ is not hyponormal if $|C|>2\sqrt{2}$.  They wonder for what values of $C$ the operator $T_{z+C|z|^2}$ is hyponormal.  We will consider the more general class of $T_{\varphi}$ when $\varphi(z)=z^n+C|z|^s$, where $n\in\bbN$, $s\in(0,\infty)$, and $C\in\bbC$.  As a consequence of our results, we will see that $T_{z+C|z|^2}$ is hyponormal if and only if $|C|\leq\frac{1}{2}$ (see Theorem \ref{better1}).  Our most general result is the following theorem.

\begin{theorem}\label{better1}
Suppose $C\in\bbC$, $s\in(0,\infty)$, and $n\in\bbN$.  If $T_{z^n+C|z|^{s}}$ is hyponormal, then $|C|\leq\frac{n}{s}$.  If $s\geq2n$, then the converse is also true.
\end{theorem}

Before we turn to the proof of this result, let us interpret it as in \cite{FL}.  Multiplication of the symbol by a non-zero constant does not effect hyponormality, so we may instead consider the symbol $Dz^n+|z|^s$.  Theorem \ref{better1} tells us that $T_{Dz^n+|z|^s}$ is not hyponormal when $|D|$ is sufficiently \textit{small} (and non-zero).  This is very surprising since we would expect small perturbations of a self-adjoint symbol by an entire function to preserve hyponormality (see also \cite[Example 1]{FL}).

Now we can turn to the proof of Theorem \ref{better1}, which we begin with a lemma.

\begin{lemma}\label{zm}
If $k\in\bbN_0$ and $t\in(0,\infty)$, then
\[
P(z^k|z|^t)=\frac{2(k+1)}{2k+t+2}z^k
\]
\end{lemma}

\begin{proof}
One can verify by direct calculation that
\[
\left\langle z^q, \frac{2(k+1)}{2k+t+2}z^k\right\rangle=\left\langle z^q, z^k|z|^t\right\rangle
\]
for every $q\in\bbN_0$, so the desired claim follows.
\end{proof}

\begin{proof}[Proof of Theorem \ref{better1}]
We begin by recalling \cite[Chapter 2, Lemma 6]{Duren}, which states that
\[
P(z^m\bar{z}^n)=
\begin{cases}
\frac{m-n+1}{m+1}z^{m-n}\qquad & m\geq n\\
0 & m<n.
\end{cases}
\]
Using this formula and Lemma \ref{zm}, it follows that if $u=\sum_{k=0}^{\infty}u_kz^k$, then
\begin{align*}
\langle T_{z^n}u,T_{z^n}u\rangle&=\sum_{k=0}^{\infty}\frac{|u_k|^2}{k+n+1}\\
\langle T_{\bar{z}^n}u,T_{\bar{z}^n}u\rangle&=\sum_{k=n}^{\infty}\frac{k-n+1}{(k+1)^2}|u_k|^2\\
\mathrm{Re}[\langle T_{z^n}u,T_{C|z|^{s}}u\rangle-\langle T_{\bar{z}^n}u,T_{\bar{C}|z|^{s}}u\rangle]&=\sum_{k=0}^{\infty}\frac{2mn\,\mathrm{Re}[u_k\bar{u}_{k+n}\bar{C}]}{(k+n+1)(2k+2n+m+2)(2k+m+2)}
\end{align*}

We then use (\ref{expand}) to conclude that
\[
\langle [T_{z^n+C|z|^{s}}^*,T_{z^n+C|z|^{s}}]u,u\rangle\geq0
\]
if and only if
\[
\sum_{k=0}^{n-1}\frac{|u_k|^2}{k+n+1}+\sum_{k=n}^{\infty}\frac{n^2|u_k|^2}{(k+1)^2(k+n+1)}+\sum_{k=0}^{\infty}\frac{4sn\,\mathrm{Re}[u_k\bar{u}_{k+n}\bar{C}]}{(k+n+1)(2k+2n+s+2)(2k+s+2)}\geq0
\]
From this expression we see that we may choose the sequence $\{u_k\}$ so that $\mathrm{Re}[u_k\bar{u}_{k+n}\bar{C}]=-|Cu_ku_{k+n}|$ for all $k\geq0$.  Thus, we may without loss of generality assume that each $u_k$ is real and non-negative, in which case the left-hand side of the above inequality is minimized by a negative $C$.  It follows that we may assume $C$ is positive and work with the condition
\[
\sum_{k=0}^{n-1}\frac{u_k^2}{k+n+1}+\sum_{k=n}^{\infty}\frac{n^2u_k^2}{(k+1)^2(k+n+1)}\geq4snC\sum_{k=0}^{\infty}\frac{u_ku_{k+n}}{(k+n+1)(2k+2n+s+2)(2k+s+2)}
\]
Since this must be true for all suitable sequences $\{u_k\}$, we set
\begin{equation}\label{utry}
u_k=
\begin{cases}
k+1,\qquad\qquad &A\leq k\leq B,\\
0 & \mathrm{otherwise}
\end{cases}
\end{equation}
where $A$ and $B$ are constants to be determined later, but we assume $A\geq n$.  Then the above inequality becomes
\[
\sum_{k=A}^{B}\frac{n^2}{(k+n+1)}\geq snC\sum_{k=A}^{B-n}\frac{k+1}{(k+n+1+s/2)(k+1+s/2)}
\]
If we set $A=t$ and $B=t^2$ for a very large integer $t$ (which we will eventually send to infinity), then we get
\[
n^2\log(t)+O(1)\geq snC\log(t)+O(1),\qquad\qquad t\rightarrow\infty.
\]
Dividing through by $\log(t)$ and sending $t\rightarrow\infty$ shows that $|C|\leq\frac{n}{s}$ is a necessary condition for hyponormality of $T_{z^n+C|z|^s}$.

To prove the statement about sufficiency in the case $s\geq2n$, notice that our above calculations show that $T_{z^n+C|z|^s}$ is hyponormal if and only if
\begin{align}\label{cinf}
|C|\leq\inf\left\{\frac{\sum_{k=0}^{n-1}\frac{u_k^2}{k+n+1}+\sum_{k=n}^{\infty}\frac{n^2u_k^2}{(k+1)^2(k+n+1)}}{sn\sum_{k=0}^{\infty}\frac{u_ku_{k+n}}{(k+n+1)(k+1+s/2)(k+n+1+s/2)}}\right\},
\end{align}
where the infimum is taken over all sequences $\{u_k\}$ such that $u_k\geq0$ for each $k$ and
\[
\sum_{k=0}^{\infty}\frac{u_k^2}{k+1}<\infty
\]
Apply the Cauchy-Schwarz inequality to the denominator inside the infimum to see
\begin{align*}
&\sum_{k=0}^{\infty}\frac{u_ku_{k+n}}{(k+n+1)(k+1+s/2)(k+n+1+s/2)}\\
&\,\leq\sqrt{\sum_{k=0}^{\infty}\frac{u_k^2}{(k+n+1)(k+1+s/2)(k+n+1+s/2)}\cdot\sum_{k=n}^{\infty}\frac{u_{k}^2}{(k+1)(k-n+1+s/2)(k+1+s/2)}}
\end{align*}
Now, the fact that $s\geq2n$ implies
\begin{align*}
\frac{1}{(k+n+1+s/2)(k+1+s/2)}&\leq\frac{1}{n^2}\qquad\qquad\quad\qquad\qquad\qquad k=0,1,\ldots,n-1\\
\frac{1}{(k+1+s/2)(k-n+1+s/2)}&\leq\frac{1}{(k+1)(k+n+1)}\qquad\qquad k=n,n+1,\ldots\\
\frac{1}{(k+1+s/2)(k+n+1+s/2)}&\leq\frac{1}{(k+1)^2},\qquad\qquad\qquad\qquad k=n,n+1,\ldots
\end{align*}
all of which can be checked by elementary calculation.  Therefore, we conclude that
\begin{align*}
&\sum_{k=0}^{\infty}\frac{u_ku_{k+n}}{(k+n+1)(k+1+s/2)(k+n+1+s/2)}\\
&\qquad\qquad\qquad\qquad\qquad\qquad\qquad\leq\sum_{k=0}^{n-1}\frac{u_k^2}{n^2(k+n+1)}+\sum_{k=n}^{\infty}\frac{u_k^2}{(k+1)^2(k+n+1)}
\end{align*}
It follows that the infimum in \eqref{cinf} is at least $\frac{n}{s}$ and hence $|C|\leq\frac{n}{s}$ is sufficient to guarantee the hyponormality of $T_{z^n+C|z|^s}$ in the case $s\geq2n$.
\end{proof}

\noindent\textit{Remark.}  Notice that the extremal problem posed in \eqref{cinf} is similar to that considered in \cite{BKLSS}, but is not identical, so the results of that paper cannot be directly applied.

\begin{corollary}\label{bigm}
If $n\in\bbN$ and $C\in\bbC\setminus\{0\}$, then there exists an $s\in(0,\infty)$ such that $T_{z^n+C|z|^s}$ is not hyponormal.
\end{corollary}

\medskip

We do not know exactly what happens when $s<2n$, but the following example shows that the conclusion of Theorem \ref{better1} does not extend to all pairs $(s,n)$.

\subsection{Example: The case $s<2n$.}\label{ex2}

By considering the case $n=7$ and $s=1$, we will see that the conclusion of Theorem \ref{better1} cannot be strengthened to include all pairs $(n,s)$.  Indeed, the operator $T_{z^{7}+C|z|}$ is hyponormal if and only if
\begin{align}\label{ccup}
|C|&\leq\inf\left\{\frac{\sum_{k=0}^5\frac{u_k^2}{k+8}+\sum_{k=6}^{\infty}\frac{49u_k^2}{(k+1)^2(k+8)}}{7\sum_{k=0}^{\infty}\frac{u_ku_{k+7}}{(k+8)(k+8+1/2)(k+1+1/2)}}\right\}
\end{align}
where the infimum is taken over all sequences $\{u_k\}$ such that $u_k\geq0$ for each $k$ and
\[
\sum_{k=0}^{\infty}\frac{u_k^2}{k+1}<\infty
\]
If the infimum (\ref{ccup}) is exactly $7$, then we recover the upper bound of $n/s$ in this case.  However, if we define
\[
u_k=
\begin{cases}
1\qquad & 0\leq k\leq 20\\
0 & k>20.
\end{cases}
\]
then the quantity inside the braces in (\ref{ccup}) is $6.41441\ldots$.  Since this value is smaller than $7$, we conclude that there are complex numbers $C$ such that $|C|<7$ and $T_{z^{7}+C|z|}$ is not hyponormal. 



\section{Additive Perturbations}\label{additive}

In this section we will explore Toeplitz operators of the form $T_{f+g}$ and address question (Q-IV).  Specifically, we will consider $T_{\varphi}$, where $\varphi(z)=z^n|z|^s+az^m|z|^t$, where $a\in\bbC$; $m,n\in\bbN$; $m<n$; and $s,t\in[0,\infty)$.  We first need the following version of Lemma \ref{zm}.

\begin{lemma}\label{zm2}
If $j,k\in\bbN_0$ and $t\in[0,\infty)$, then
\[
P(z^k\bar{z}^j|z|^t)=\begin{cases}
0 \qquad\qquad & j>k\\
\frac{2(k-j+1)}{2k+t+2}z^{k-j} & j\leq k
\end{cases}
\]
\end{lemma}

\begin{proof}
The case $j\leq k$ follows from Lemma \ref{zm}.  If $j>k$, then $z^k\bar{z}^j|z|^t\perp z^m$ for all $m\in\bbN_0$, so the projection to the Bergman space is $0$.
\end{proof}

Define
\begin{align*}
\sigma_k&=
\begin{cases}
\frac{4(k+n+1)}{(2(n+k)+s+2)^2}\qquad\qquad\qquad\qquad\qquad\qquad\qquad\qquad\qquad & 0\leq k< n\\
\frac{4(k+n+1)}{(2(n+k)+s+2)^2}-\frac{4(k-n+1)}{(2k+s+2)^2} & k\geq n
\end{cases}
\\
\omega_k&=
\begin{cases}
\frac{4(k+m+1)}{(2(m+k)+t+2)^2}\qquad\qquad\qquad\qquad\qquad\qquad\qquad\qquad\qquad & 0\leq k< m\\
\frac{4(k+m+1)}{(2(m+k)+t+2)^2}-\frac{4(k-m+1)}{(2k+t+2)^2} & k\geq m
\end{cases}
\\
\delta_k&=
\begin{cases}
\frac{4(k+n+1)}{(2(k+n)+s+2)(2(k+n)+t+2)}\qquad\qquad\qquad\qquad\qquad\qquad\qquad & 0\leq k<m\\
\frac{4(k+n+1)}{(2(k+n)+s+2)(2(k+n)+t+2)}-\frac{4(k-m+1)}{(2(k+n-m)+s+2)(2k+t+2)} & k\geq m
\end{cases}
\end{align*}
It is an elementary calculation to verify that each of the sequences $\{\sigma_k\}$, $\{\omega_k\}$, and $\{\delta_k\}$ contains only positive real numbers.

Mimicking the calculations in the proof of \cite[Theorem 4]{FL}, one finds the operator $T_{\varphi}$ (with $\varphi$ defined as at the beginning of this section)  is hyponormal if and only if
\[
|a|^2\sum_{k=0}^{\infty}\omega_k|u_k|^2+2\sum_{k=0}^{\infty}\delta_k\Real[\bar{a}u_k\bar{u}_{k+n-m}]+\sum_{k=0}^{\infty}\sigma_k|u_k|^2\geq0.
\]
for all sequences $\{u_k\}$ such that $\sum u_kz^k\in A^2(\bbD)$.  We are free to choose the phases of each $u_k$ so we may assume that every term in the middle sum is real and negative (to give us the worst case scenario).  Thus, we rewrite the above condition as
\begin{equation}\label{aeq}
Q_{u}(|a|):=|a|^2\sum_{k=0}^{\infty}\omega_k|u_k|^2-2|a|\sum_{k=0}^{\infty}\delta_k|u_k\bar{u}_{k+n-m}|+\sum_{k=0}^{\infty}\sigma_k|u_k|^2\geq0.
\end{equation}
Notice that $Q_{u}$ is a quadratic polynomial.  The statement that $T_{\varphi}$ is hyponormal is equivalent to the statement that $|a|$ does not lie between the real roots of $Q_{u}$ for any choice of $\{u_k\}$.  
We have proven the following result.

\begin{prop}\label{better3a}
If $\varphi(z)=z^n|z|^s+az^m|z|^t$, where $a\in\bbC$; $m,n\in\bbN$; $m<n$; and $s,t\in[0,\infty)$, then $T_{\varphi}$ is hyponormal if and only if $|a|$ never lies between
\[
\frac{\sum_{k=0}^{\infty}\delta_k|u_k\bar{u}_{k+n-m}|-\sqrt{\left(\sum_{k=0}^{\infty}\delta_k|u_k\bar{u}_{k+n-m}|\right)^2-\sum_{k=0}^{\infty}\omega_k|u_k|^2\sum_{k=0}^{\infty}\sigma_k|u_k|^2}}{\sum_{k=0}^{\infty}\omega_k|u_k|^2}
\]
and
\[
\frac{\sum_{k=0}^{\infty}\delta_k|u_k\bar{u}_{k+n-m}|+\sqrt{\left(\sum_{k=0}^{\infty}\delta_k|u_k\bar{u}_{k+n-m}|\right)^2-\sum_{k=0}^{\infty}\omega_k|u_k|^2\sum_{k=0}^{\infty}\sigma_k|u_k|^2}}{\sum_{k=0}^{\infty}\omega_k|u_k|^2}
\]
for any $\{u_k\}_{k=0}^{\infty}$ satisfying $\sum |u_k|^2/(k+1)<\infty$.
\end{prop}

\noindent\textit{Remark.}  If the quantity under the square roots in Proposition \ref{better3a} is negative for some particular choice of sequence $\{u_k\}$, then that sequence places no constraint on $|a|$.  

\medskip

We see that for every suitable sequence $\{u_k\}$ we obtain an (possibly empty) open annulus such that if $a$ lies in this annulus, then $T_{\varphi}$ is not hyponormal.  Proposition \ref{better3a} states that $T_{\varphi}$ is hyponormal if and only if $a$ lies outside the union of all of these annuli.  It is not obvious how to describe this uncountable union of annuli, though the following result sheds some light on the situation. 

\begin{theorem}\label{better3b}
If $\varphi(z)$ is as in Proposition \ref{better3a} with $m$, $n$, $s$, and $t$ fixed, then there exist values of $a\in\bbC$ for which $T_{\varphi}$ is hyponormal.  If we further assume that $ms\neq nt$, then there exist values of $a\in\bbC$ for which $T_{\varphi}$ is not hyponormal.
\end{theorem}

\begin{proof}
First we will prove that there exist values of $a\in\bbC$ for which $T_{\varphi}$ is hyponormal.  Notice that there are positive constants $C_j$ for $j=1,\ldots,6$ such that
\[
\frac{C_1}{(k+1)^3}\leq|\sigma_k|\leq\frac{C_2}{(k+1)^3},\qquad \frac{C_3}{(k+1)^3}\leq|\omega_k|\leq\frac{C_4}{(k+1)^3},\qquad \frac{C_5}{(k+1)^3}\leq|\delta_k|\leq\frac{C_6}{(k+1)^3}
\]
Let
\[
C_+=\max\{C_1,\ldots,C_6\}\qquad\qquad C_-=\min\{C_1,\ldots,C_6\}.
\]
Then the Cauchy-Schwarz inequality implies
\begin{align*}
Q_{u}(|a|)&\geq|a|^2\left(C_-\sum_{k=0}^{\infty}\frac{|u_k|^2}{(k+1)^3}\right)+C_-\sum_{k=0}^{\infty}\frac{|u_k|^2}{(k+1)^3}\\
&\qquad\qquad\qquad-\frac{2|a|C_+(n-m+1)^{3/2}}{C_-}\left(C_-\sum_{k=0}^{\infty}\frac{|u_k|^2}{(k+1)^3}\right)\\
&=\left(C_-\sum_{k=0}^{\infty}\frac{|u_k|^2}{(k+1)^3}\right)\left(|a|^2-\frac{2|a|C_+(n-m+1)^{3/2}}{C_-}+1\right)
\end{align*}
which is positive for all non-trivial choices of $\{u_k\}$ as long as $|a|$ is sufficiently large or sufficiently small.

To prove the second claim, we use a trial vector as in \eqref{utry} (with $A=x$ and $B=x^2$) to show that the bounds given by Proposition \ref{better3a} are positive and unequal.  To complete the calculation, we note that
\begin{align*}
\sigma_k&=\frac{n(n+s)}{k^3}+\mco(k^{-4}),\qquad\qquad k\rightarrow\infty,\\
\omega_k&=\frac{m(m+t)}{k^3}+\mco(k^{-4}),\qquad\qquad k\rightarrow\infty,\\
\delta_k&=\frac{mn+\frac{ms+nt}{2}}{k^3}+\mco(k^{-4}),\qquad\qquad k\rightarrow\infty.
\end{align*}
Using these formulas, we have (as $x\rightarrow\infty$)
\begin{align*}
&
\frac{\sum_{k=0}^{\infty}\delta_k|u_k\bar{u}_{k+n-m}|\pm\sqrt{\left(\sum_{k=0}^{\infty}\delta_k|u_k\bar{u}_{k+n-m}|\right)^2-\sum_{k=0}^{\infty}\omega_k|u_k|^2\sum_{k=0}^{\infty}\sigma_k|u_k|^2}}{\sum_{k=0}^{\infty}\omega_k|u_k|^2}\\
&=\frac{\frac{2mn+ms+nt}{2}\log(x)+\mco(1)\pm\log(x)\sqrt{(mn+\frac{ms+nt}{2})^2-mn(n+s)(m+t)+\mco(1)}}{m(m+t)\log(x)+\mco(1)}\\
&=\frac{2mn+ms+nt\pm|ms-nt|}{2m(m+t)}+o(1)
\end{align*}
Thus, as long as $ms\neq nt$ it is true that
\[
|a|\in\left(\frac{2mn+ms+nt-|ms-nt|}{2m(m+t)},\frac{2mn+ms+nt+|ms-nt|}{2m(m+t)}\right)
\]
implies $T_{\varphi}$ is not hyponormal.
\end{proof}

\noindent\textit{Remark.}  We can actually refine the statement of Theorem \ref{better3b} by using the notation defined in its proof.  Indeed, by using that notation we can say that if $|a|$ is not in the interval
\begin{align*}
\left[(n-m+1)^{3/2}\frac{C_+}{C_-}-\sqrt{\frac{C_+^2(n-m+1)^3}{C_-^2}-1},(n-m+1)^{3/2}\frac{C_+}{C_-}+\sqrt{\frac{C_+^2(n-m+1)^3}{C_-^2}-1}\right]
\end{align*}
then $T_{\varphi}$ is hyponormal.  This condition is a substantial improvement to the conditions for hyponormality given in \cite[Remark after Theorem 4]{FL}.

\medskip

We can use the same ideas as above to prove an extension of \cite[Theorem 5]{FL}.  Define
\begin{align*}
\sigma_k'&=
\begin{cases}
\frac{4(k+n+1)}{(2(n+k)+s+2)^2}\qquad\qquad\qquad\qquad\qquad\qquad\qquad\qquad\qquad & 0\leq k< n\\
\frac{4(k+n+1)}{(2(n+k)+s+2)^2}-\frac{4(k-n+1)}{(2k+s+2)^2} & k\geq n
\end{cases}
\\
\omega_k'&=
\begin{cases}
-\frac{4(k+m+1)}{(2(m+k)+t+2)^2}\qquad\qquad\qquad\qquad\qquad\qquad\qquad\qquad\qquad & 0\leq k< m\\
\frac{4(k-m+1)}{(2k+t+2)^2}-\frac{4(k+m+1)}{(2(m+k)+t+2)^2} & k\geq m
\end{cases}
\\
\delta_k'&=
\begin{cases}
\frac{4(k+n+1)}{(2(k+n)+s+2)(2(k+n+m)+t+2)}-\frac{4(k+m+1)}{(2(k+n+m)+s+2)(2(k+m)+t+2)}
\end{cases}
\end{align*}
It is an elementary calculation to see that $\omega_k'<0$ for all $k$ and $\sigma_k>0$ for all $k$.  
With this notation and nearly the same calculations as above, we can prove our next result.

\begin{prop}\label{better4a}
If $\varphi(z)=z^n|z|^s+a\bar{z}^m|z|^t$, where $a\in\bbC$; $m,n\in\bbN$; $s,t\in[0,\infty)$; and $\delta_k'>0$ for all $k$, then $T_{\varphi}$ is hyponormal if and only if
\[
|a|\leq\frac{\sum_{k=0}^{\infty}|\delta_k'u_k\bar{u}_{k+n+m}|-\sqrt{\left(\sum_{k=0}^{\infty}|\delta_k'u_k\bar{u}_{k+n+m}|\right)^2-\sum_{k=0}^{\infty}\omega_k'|u_k|^2\sum_{k=0}^{\infty}\sigma_k'|u_k|^2}}{\sum_{k=0}^{\infty}\omega_k'|u_k|^2}
\]
for every $\{u_k\}_{k=0}^{\infty}$ satisfying $\sum |u_k|^2/(k+1)<\infty$.
\end{prop}

One can then mimic the calculations in the proof of Theorem \ref{better3b} to deduce that if $m,n,s,t$ are fixed as in Proposition \ref{better4a}, then for sufficiently small values of $|a|$, the operator $T_{\varphi}$ is hyponormal while it is not hyponormal when $|a|$ is sufficiently large.  In the special case when $m=n$ and $s=t$, then $\delta_k'\equiv0$ and Proposition \ref{better4a} simplifies to the following statement.

\begin{corollary}\label{better4b}
If $n\in\bbN$ and $s\in[0,\infty)$, then $T_{|z|^s(z^n+a\bar{z}^n)}$ is hyponormal if and only if $|a|\leq1$.
 \end{corollary}
 
 Thus we see that $T_{|z|^s(z^n+a\bar{z}^n)}$ is hyponormal precisely when the analytic part dominates the anti-analytic part of the symbol.  The special case of Corollary \ref{better4b} when $s=0$ is contained in \cite[Proposition 1.4.4]{Sadraoui}.

\section{Multiplicative Perturbations}\label{multiplicative}

In this section, we will list some generalizations of results from \cite{FL} that are related to question (Q-V).  Many of these results share similar (or even identical) proofs to the corresponding results from \cite{FL}, but we state them here with added generality for the sake of completeness.

\subsection{The Operator $T_{z^n|z|^s}$}\label{th2}

Let us begin by proving the following generalization of \cite[Theorem 2]{FL}, which considers the symbol $T_{\varphi}$ with $\varphi$ as in Section \ref{additive} but with $n=m$ and $s=t$.

\begin{theorem}\label{better2}
If $n\in\bbN$ and $s\in(0,\infty)$, then $T_{z^n|z|^s}$ is hyponormal.
\end{theorem}

\begin{proof}
It suffices to prove Proposition \ref{better3a} in the case $m=n$ and $s=t$.  In this case, $\sigma_k=\omega_k=\delta_k$ and so hyponormality is equivalent to
\[
|1+a|^2\sum_{k=0}^{\infty}\omega_k|u_k|^2\geq0,
\]
which is true.
\end{proof}

This result should be interpreted in contrast to Theorem \ref{better1}.  That result says that the self-adjoint (and hence hyponormal) operator $T_{|z|^s}$ is transformed into an operator that is not hyponormal by the addition of $\epsilon z$ to the symbol for sufficiently small $\epsilon>0$.  Theorem \ref{better2} states that a multiplicative perturbation by the same function does not destroy hyponormality.


\subsection{Algebraic Functions of Fixed Relative Degree}\label{fixed}

Here we present an improvement of \cite[Theorem 14]{FL}.

\begin{theorem}\label{rel1}
Suppose $\varphi(z)=z^m\left(|z|^{s_0}+a_1|z|^{s_1}+\cdots+a_n|z|^{s_n}\right)$ where $m,n\in\bbN$ and $s_i\in(0,\infty)$ for all $i=0,\ldots,n$.  Then $T_{\varphi}$ is hyponormal if and only if
\begin{equation}\label{alph}
\left|\sum_{j=0}^n\frac{a_j}{2\alpha+2m+2+s_j}\right|^2\geq\frac{\alpha-m+1}{\alpha+m+1}\left|\sum_{j=0}^n\frac{a_j}{2\alpha+2+s_j}\right|^2
\end{equation}
for all natural numbers $\alpha\geq m$, where $a_0=1$.
\end{theorem}

\begin{proof}
This is an immediate consequence of \cite[Theorem 3.1]{LuLiu}.
\end{proof}

The nice thing about Theorem \ref{rel1} is that after expanding the squares and clearing the denominators, one obtains a positivity condition on a polynomial in $\alpha$ on the set $[m,\infty)\cap\bbN$.  From a practical standpoint, this is a straightforward condition to verify numerically.

If $n=1$, then we can make much more explicit conclusions that are related to question (Q-V).  In this case, the condition \eqref{alph} is equivalent to a polynomial $\mcp(\alpha)$ being positive on $[m,\infty)\cap\bbN$, where the degree of $\mcp$ is at most $5$.  Therefore, one can express the critical points of $\mcp$ as algebraic functions of the coefficients and thus find elementary conditions on $m$, $s_0$, $s_1$, and $a_1$ that guarantee the positivity of $\mcp$ on $[m,\infty)\cap\bbN$.  An example of the kind of conclusions one can reach in this fashion is given by the following corollary.

\begin{corollary}\label{better10}
Suppose $\varphi(z)=z^m\left(|z|^{s_0}+a_1|z|^{s_1}\right)$ where $m\in\bbN$ and $s_0,s_1\in(0,\infty)$ satisfy $s_0\neq s_1$.
\begin{itemize}
\item[(i)]  If $a_1=-1$, then $T_{\varphi}$ is not hyponormal.
\item[(ii)]  If $a_1\neq-1$ and
\[
m+s_1<(s_1-s_0)\Real\left[\frac{1}{1+a_1}\right]
\]
then $T_{\varphi}$ is not hyponormal.
\item[(iii)]  If $\Real[a_1]>0$, then $T_{\varphi}$ is hyponormal.
\end{itemize}
\end{corollary}

\begin{proof}
As we have stated, the conclusion of Theorem \ref{rel1} states that $T_{\varphi}$ is hyponormal if and only if some polynomial $\mcp(\alpha)$ is non-negative on $[m,\infty)\cap\bbN$.  If $a_1=-1$, then $\mcp$ has degree $4$ and
\[
\mcp(\alpha)=\alpha^4\left[-\frac{m}{2}(s_0-s_1)^2\right]+\cdots
\]
Since this leading coefficient is negative, we conclude that the operator $T_{\varphi}$ is not hyponormal.  If $a_1\neq-1$, then $\mcp$ has degree $5$ and satisfies
\[
\mcp(\alpha)=\alpha^5\left[m((m+s_1)|1+a_1|^2+(s_0-s_1)(\Real[a_1]+1)\right]+\cdots
\]
so non-negativity of this leading coefficient becomes a necessary condition for hyponormality.  Negativity of this leading coefficient is precisely the statement in (ii) above.

Finally, if $\Real[a_1]>0$, then one can verify by hand that the polynomial $\mcp$ has only positive coefficients.
It follows that $\mcp$ is positive on $[0,\infty)$ and hence $T_{\varphi}$ is hyponormal.
\end{proof}

Notice that part (i) of Corollary \ref{better10} generalizes \cite[Example 11]{FL}.



\end{document}